\documentclass[12pt]{article}
\usepackage{amssymb,amsmath,amsthm}
\theoremstyle{plain}
\newtheorem{theorem}{Theorem}[section]
\newtheorem{lemma}[theorem]{Lemma}

\theoremstyle{definition}
\newtheorem{ex}{Example}
\newtheorem{qu}{Question}
\newtheorem{co}{Comment}
\newtheorem{defn}{Definition}
\newtheorem{notation}{Notation}

\newcommand\dref[1]{Definition~\ref{defn:#1}}
\newcommand\tref[1]{Theorem~\ref{thm:#1}}
\newcommand\lref[1]{Lemma~\ref{lem:#1}}
\newcommand\cref[1]{Corollary~\ref{cor:#1}}
\newcommand\coref[1]{Comment~\ref{co:#1}}

\newcommand\qref[1]{Question~\ref{qu:#1}}
\newcommand\eref[1]{Equation~\eqref{eg:#1}}
\newcommand\exref[1]{Example~\ref{ex:#1}}

\begin{document}
\newcommand\sall{$\sigma$-algebra}
\newcommand\sal{$\sigma$-algebra~}
\newcommand\sals{$\sigma$-algebras~}
\newcommand\ttt{$T,T^{-1}$~transformation~}
\newcommand\tttl{$T,T^{-1}$~transformation}
\newcommand\ttp{$T,T^{-1}$~process~}
\newcommand\ttpl{$T,T^{-1}$~process}
\newcommand\Om{$\Omega$~}
\newcommand\Oml{$\Omega$}
\newcommand\om{$\omega$~}
\newcommand\oml{$\omega$}
\newcommand\omo{$\omega_1$~}
\newcommand\omol{$\omega_1$}
\newcommand\omt{$\omega_2$~}
\newcommand\omtl{$\omega_2$}
\newcommand\omb{$\Omega$~}
\newcommand\ombl{$\Omega$}
\newcommand\proa{$...a_{-2}, a_{-1}, a_{0}, a_{1}, a_{2},...$~}
\newcommand\prob{$...b_{-2}, b_{-1}, b_{0}, b_{1}, b_{2},...$~}
\newcommand\proc{$...c_{-2}, c_{-1}, c_{0}, c_{1}, c_{2},...$~}
\newcommand\proX{$...X_{-2}, X_{-1}, X_{0}, X_{1}, X_{2},...$~}
\newcommand\proan{$a_{0}, a_{1}, a_{2},...$~}
\newcommand\probn{$b_{0}, b_{1}, b_{2},...$~}
\newcommand\procn{$c_{0}, c_{1}, c_{2},...$~}
\newcommand\procX{$X_{0}, X_{1}, X_{2},...$~}
\newcommand\proA{$...A_{-2}, A_{-1}, A_{0}, A_{1}, A_{2},...$~}
\newcommand\proB{$...B_{-2}, B_{-1}, B_{0}, B_{1}, B_{2},...$~}
\newcommand\proC{$...C_{-2}, C_{-1}, C_{0}, C_{1}, C_{2},...$~}
\newcommand\proAn{$A_{0}, A_{1}, A_{2},...$~}
\newcommand\proBn{$B_{0}, B_{1}, B_{2},...$~}
\newcommand\proCn{$C_{0}, C_{1}, C_{2},...$~}
\title{Nonintersecting splitting \sals in a non-bernoulli transformation}
\author{Steven Kalikow
\thanks{University of Memphis, Department of
Mathematics, 3725 Norriswood, Memphis, TN 38152, USA.}
}

\maketitle

{\bf This paper is dedicated to the memory of Dan Rudolph}

\begin{abstract}
Given a measure preserving transformation $T$ on a Lebesgue \sall, a complete $T$ invariant sub \sal is said to split if there is another complete $T$ invariant sub \sal on which $T$ is Bernoulli which is completely independent of the given sub \sal and such that the two sub \sals together generate the entire \sall.  It is easily shown that two splitting sub \sals with nothing in common imply $T$ to be K. Here it is shown that $T$ does not have to be Bernoulli by exhibiting two such nonintersecting \sals for the \tttl , negatively answering a question posed by Thouvenot in 1975. \end{abstract}
\section{Introduction}
\begin{notation}\label{not:1}
Throughout, any transformation $T$ that we consider is an ergodic transformation of a Lebesgue space \Om endowed with \sal $A$ and measure $\mu$. $T,\Omega,A$ and $\mu$ are assumed unless we say otherwise.
\end{notation}
\begin{defn}
\label{defn:pin}
The Pinsker algebra is the largest  $T$ invariant sub \sal $S$ of $A$ such that $T$ has entropy 0 on $S$.
\end{defn}
\begin{defn}
\label{defn:K}
$T$ is called a $K$ transformation if its Pinsker algebra is trivial (i.e. consists only of the sets whose measure is either 0 or 1.) This is equivalent to saying that $T$ has no nontrivial $0$ entropy factors (The trivial $0$ entropy factor is the unique transformation $T'$ which acts on the unique Lebesgue space consisting of only one point.)
\end{defn}
\begin{defn}
\label{defn:split}
A $T$ invariant sub \sal $B$ of $A$ is said to split if there is another $T$ invariant sub \sal $C$ of $A$ such that $T$ on $C$ is Bernoulli, $B$ is independent of $C$ and together $B$ and $C$ generate $A$.
\end{defn}
If a sub-\sal splits then it contains the Pinsker algebra. Hence if two sub-\sals both split, their intersection must contain the Pinsker algebra and thus if that intersection is trivial the process must be $K$. In \cite{OS} it was shown that there are uncountably many transformations that are $K$ and not Bernoulli. In \cite{ttt} it was shown that a particular transformation called the \ttt was one of them. 

As a result of the second sentence of the above paragraph, Thouvenot posed the following question

\begin{qu}\label{qu:ber}If there are two complete $T$ invariant sub-\sals of a given \sal 
which both split but whose intersection is trivial, does it follow that the transformation is Bernoulli?
\end{qu}

Actually this is a weakening of the question he asked, namely Question A in \cite{A} posed in 1975, namely

\vskip.5cm

\noindent {\bf Question A}. If there are two complete $T$ invariant sub-\sals $A_1$ and $A_2$ of a given \sal $A$ such that both $A_1$ and $A_2$ split in $A$, does it follow that the intersection of $A_1$ and $A_2$ splits in $A_1$?

\vskip.5cm

The reason a positive answer to \qref{ber} follows from a positive answer to Question A is that if the trivial \sal splits in $A_1$ and $A_1$ splits in $A$ then $T$ on $A_1$ is Bernoulli and hence $T$ on $A$ is Bernoulli. Thouvenot mentioned Question A again in \cite{again} where he proved that

\vskip.5cm

\noindent If $T_1$ has the weak Pinsker property (it is not necessary for the reader to know what that means in this paper) $T_B$ is Bernoulli of finite entropy and $T_1 \times T_B$ is isomorphic to $T_2 \times T_B$, then $T_1$ is isomorphic to $T_2$.

\vskip.5cm

The reason he still wanted to know the resolution of Question A was that the above result would be rendered easy for any $T_1$ and $T_2$ if $T_1 \times T_B$ had finite entropy and the statement of Question A were valid. 
\begin{proof}
From the preconditions it would follow that $T_1$ has the same entropy as $T_2$. Let $\Omega_1,\Omega_2,\Omega_3,\Omega_4$, and $\Omega_5$ respectively be the spaces that $T_1$ acts on, $T_2$ acts on, $T_B$ acts on, $\Omega_1 \times \Omega_3$, and $\Omega_2 \times \Omega_3$. If we look at the isomorphism from $T_1 \times T_B$ to $T_2 \times T_B$ (from $\Omega_4$ to $ \Omega_5$)
let $A_2$ be the embedding of the \sal generated by $T_2$ in $\Omega_2$ into $\Omega_5$.
Let $A_1$ be the isomorphic image of [the inbedding of \sal generated by $T_1$ into $\Omega_4$] into $\Omega_5$. 
Let $T$ be $T_2 \times T_B$ (which acts on $\Omega_5$).
Let $A$ be the entire \sal generated by $T$ in $\Omega_5$.
We have that $A_1,A_2$,and $A$ are all on $\Omega_5$ and under $T$, both $A_1$ and $A_2$ split in $A$ and hence a positive result to Question A would give that their intersection splits in both $A_1$ and $A_2$ under $T$. The restriction of $T$ to $A_1$ gives a transformation isomorphic to $T_1$ which can be written as the cross of $T$ restricted to that intersection with a Bernoulli and the restriction of $T$ to $A_2$ gives a transformation isomorphic to $T_2$ which can be written as the cross of $T$ restricted to that intersection with another Bernoulli. Since $T_1$ has the same entropy as $T_2$ the two Bernoullis have the same entropy and hence are isomorphic making $T_1$ isomorphic to $T_2$.
\end{proof}

The purpose of this paper is to answer \qref{ber} negatively (thereby answering Question A negatively) by showing that we can do this with the \tttl. But it would be unfair to give this paper complete credit for answering \qref{ber}. The two $T$ invariant sub-\sals we will give will be called $B$ and $C$ to be defined later. The creation of $B$ and the fact that it split was established by Thouvenot. He never actually wrote this in a paper but he made it general knowledge. The creation of $C$ and the fact that it split was established by Hoffman in \cite{hoff}. The reason why \qref{ber} has until now been too hard to solve is that Hoffman's result was not yet known. After Hoffman established that $C$ splits it was clear that if someone could show that $B$ and $C$ had trivial intersection \qref{ber} would be resolved. Hoffman himself never had the opportunity of resolving this issue because he was not aware of the question. Our contribution here is to establish that $B$ and $C$ has trivial intersection, thereby completing the resolution of \qref{ber} negatively. Actually the \sal $B$ that we use here is not precisely the \sal that Thouvenot proved to split but $B$ is just a small modification of that \sal and the proof that it splits is identical. Here we will state the modified Thouvenot \sal $B$, show that $B$  and $C$ have trivial intersection, and then give the proof (essentially Thouvenot's) that $B$ splits. 

Let me be more precise. Thouvenot starts with a set to make his construction. He insists that the set be a stretch of the path (we will later explain what path means). However we will show that he could have let it be a stretch of both the path and the scenery (we will later explain what scenery means). We use exactly the Thouvenot construction to get $B$ except that we use a set that depends on more than just the path. For that reason we include in the last section a proof (essentially Thouvenot's) that $B$ splits. 

Thus we are faced with the following tasks in this paper. 
\\1)	Define the \ttt (Although this has already been done in \cite{ttt} we do it again here to make this paper self contained.)
\\2)	Define $B$.
\\3)	Define $C$. 
\\4)	Show that $B$ and $C$ have trivial intersection.
\\5)	Use Thouvenot's proof to show that $B$ splits.

\section{Acknowledgements} I am grateful to Jean Paul Thouvenot for introducing me to the problem, mentioning his and Hoffman's \sall s to me and indicating to me that I could solve the problem if I could prove their intersection to be trivial. This is just one of the many examples of the enormous amount of support I have received from him in the past 15 years. During those years he took it upon himself to call me about twice a week without my asking him to do this and since I live alone this has been enormously supportive. I would also like to thank the referee whose critisms caused this to be a much clearer paper than it would have been otherwise. 

\section{Dedication to Dan Rudolph}

Everyone understands that Dan was the best ergodic theorist of our generation in the world but that is not the purpose of this dedication. The purpose of this dedication is to indicate the kind of person Dan was as a person. While we shared a house together in Berkeley with about three other people he did most the work of organizing and supporting that house without seeming to in any way resent doing that. When I got sick and was afraid I would not be able to handle a job, he and Ken Berg encouraged me to come to Maryland anyway and with Ken took on my work load for a semester when I was hospitalized. While he was chairman of the department and a father and a thesis advisor he still found time to sit in on my ergodic theory course as a service to me (he obviously did not need to learn ergodic theory from me). One of my criteria for judging a person is whether he tends to say bad things about others. Dan never did. He was always positive in his impressions of others always looking for the good in others.

\section{ General Theory}
~~~~~Since this paper makes heavy use of the concept of fibers and since we fear that many readers are uncomfortable about fiber arguments, we wish to give a discussion about fiber arguments and why they are valid. We will usually omit proofs of theorems in this section as we are only discussing theory that the reader should already be familiar with. When we do give proofs in this section they will be only sparse proofs as the reader should be able to fill in the details himself. As throughout this paper all transformations $T$ are on a Lebesgue space and $T$ acts on a \sal $A$.

In this paper all transformations considered have finite entropy and all such transformations have a finite generator.

\begin{defn} \label{defn:qtname} Let $T$ be a transformation, $Q$ be a partition and \om be a point in the space. The $Q,T$ name of \om is the sequence \proa such that $a_i$ is the element of $Q$ which $T^i$(\oml) is in.
\end{defn} 

\begin{notation}\label{not:A1A2} $A_1$ and $A_2$ are complete sub-\sall s of $A$.
\end{notation}

\begin{notation}\label{not:symdif} $\Delta$ means symmetric difference.
\end{notation}

\begin{defn} \label{defn:lit} The literal \sal generated by $Q,T$ is the smallest \sal containing all sets in the partitions $\bigvee\limits_{-n}^{n}(T^i(Q))$ for all $n$.  The \sal generated by $Q,T$ is all sets of the form $S_1 \Delta Z$ where $S_1$ is in the literal \sal generated by $Q,T$ and $Z$ has measure $0$.  \end{defn} 

\begin{theorem}\label{samename} Let $T_1$ and $T_2$ be two measure preserving transformations acting on the same space and let $Q_1$ and $Q_2$ be two partitions. Suppose the \sal generated by $Q_1, T_1$ and $Q_2, T_2$ both equal $A_1$. Then there exists a set $Z$ of measure $0$ such that if \omo and \omt are not in $Z$ then they have the same $Q_1, T_1$ name iff they have the same $Q_2, T_2$ name.
\end{theorem}

\begin{proof}Let $S \in A_1$. You can approximate $S$ arbitrarily well with a finite union sets in the partition $\bigvee\limits_{-n}^{n}(T_1^i(Q_1))$  for some $n$. It follows that after removing a set of measure $0$, the $T_1,Q_1$ name of \omo determines whether or not you are in $S$. By letting $S$ be a set in $\bigvee\limits_{-n}^{n}(T_2^i(Q_2))$, we get that after removing  a set of  measure 0 the $T_1,Q_1$ name of  \omt determines the $T_2,Q_2$ name of \omt from -n to n and hence (after removing set of measure $0$) you can get it to determine the entire $T_2,Q_2$ name of \omt. Argue symmetrically to get the converse
\end{proof}.

\begin{co}\label{co:infinite}As a matter of fact there is no reason to restrict to finite entropy because the above theorem is still true if one or both of those partitions are countably infinite. In the proof you still need to get an increasing sequence of finite partitions which generate $A_1$ but this can be done by truncating your countable generator to make it finite at any finite stage. 
\end{co}  

The above theorem makes the following definition well defined up to measure $0$, i.e. it does not depend on what transformation $T$ we use or what generator $Q$ we use. 

\begin{defn} \label{defn:fiber} Let $Q,T$ generate $A_1$. Then \omo and \omt are said to be in the same fiber of $A$ over $A_1$ if they have the same $Q,T$ name. Being in the same fiber of $A$ over $A_1$ is an equivalence relation and a fiber of $A$ over $A_1$ is an equivalence class for that equivalence relation. We can just say fiber of $A_1$ or same fiber of $A_1$ if $A$ is understood.
\end{defn}

\begin{theorem}\label{bnlit} Let $S$ be a set in the literal \sal generated by $Q,T$. Then if \omo and \omt have the same $Q,T$ name, either they are both in $S$ or neither of them is in $S$.
\end{theorem} 

\begin{proof}  Just show that it is true for all sets in the appropriate partitions and that the collection of sets for which it is true is a  \sall. 
\end{proof}

From this it follows that 

\begin{theorem}\label{thm:bn} Let $S$ be a set in $A_1$. There exists a set $Z$ of measure $0$ dependent on what set $S$ we use and on what transformation and generator we used to define fibers, such that if \omo and \omt are in same fiber of  $A$ over $A_1$ and neither \omo nor \omt are in $Z$, then either they are both in $S$ or neither of them is.
\end{theorem} 

\begin{defn} \label{defn:2point} We say $A_2$ is a two point extension of $A_1$ if there is a set $Z$ of measure $0$ such that off of $Z$, each fiber of $A$ over $A_1$ is a union of two fibers of $A$ over $A_2$.
\end{defn}

\section{Definitions of the \tttl, $B$ and $C$}

\begin{defn} \label{defn:tt-1} 
Let \proa (called the scenery) be independent tosses of a fair coin and let
\prob (called the path) also be independent tosses of a fair coin. The path and scenery are chosen independently of each other. The $a_i$ take on the values $h$ or $t$ and the $b_i$ each take on the values $L$ or $R$ but the distribution of both processes are the same ($1/2,1/2$ product measure). $h$ connotes heads $t$ connotes tails $L$ connotes left and $R$ connotes right. The \ttt is a stationary process on a four letter alphabet $\{(h,L),(t,L),(h,R),(t,R)\}$. 
To generate a word in the \ttt we take a random walk on \proa using the \prob to tell you how to walk. 
\end{defn}

\begin{ex}\label{ex:tti}Suppose the terms $b_{-2},b_{-1},b_{0},b_{1},b_{2}$ take on the values $L,L,L,L,R$ and 
$a_{-2},a_{-1},a_{0},a_{1},a_{2}$ take on the values $h,h,t,t,t$. Here is how we start to generate a word
\proc in the \tttl . Since $b_0 = L$ and $a_0 = t, c_0=(t,L)$. 
Now since $b_0 = L$ it means we walk to the left on the scenery. On the path the 0 coordinate always goes to the right (i.e. the sequence \prob always shifts to the left). Since $b_1=L$ and $a_{-1}=h,~c_1 = (L,h)$. Now since the first coordinate of $c_{1}$ is $L$, we walk to the left again on the scenery. Since $b_2 = R$ and $a_{-2}= h,~c_2 =(h,R)$. We also go backwards in time to get $c_{-1},c_{-2}~$ etc. Since $b_{-1} =L$ we just walked left and since we start at 0 we must have been at position 1 at time -1 and so since $a_1 = t$ and $b_{-1}=L, ~c_{-1}=(L,t)$.
\end{ex}
\begin{defn}\label{defn:sp} In the above definition, \proa is called the scenery and \prob is called the path. 
\proX is the \ttp where each $X(i)$ takes on the value 
\\ $c_i \in P:=\{(h,L),(h,R),(t,L),(t,R)\}$ where $c_i$ is as in \exref{tti}.
If you use the \ttp as a measure on doubly infinite words in alphabet $P$, then taking a doubly infinite word and shifting it to the right is the \ttt.
\end{defn}

\begin{co}\label{co:eas} There is an easier way to define the \ttt which explains why it is called the \tttl. Let S be the transformation on doubly infinite sequences with i.i.d. $1/2,~ 1/2$ probability (in this case the \prob process) which shifts a word to the left and $T$ be an independent transformation on doubly infinite sequences with i.i.d. $1/2,~ 1/2$  probability (in this case the \proa process) which shifts a word to the left.  Partition the space into sets; Left = ($b_0 = L$) and 
Right = ($b_0 = R$). The \ttt is the transformation on the product space which takes $(\omega_1,\omega_2)$ to $(S(\omega_1),T(\omega_2))$ if $\omega_1$ is in right and $(S(\omega_1),T^{-1}(\omega_2))$ if $\omega_1$ is in left and then the \ttp is the (\ttt, $P$) process where $P$ is as defined above, namely the four set partition determined by whether $a_0$ is heads or tails and whether you are in the right or left.
\end{co}

\begin{defn}\label{defn:fN} 
\begin{equation}
f(i)~~:=~~\begin{cases} 1~~~ \text{if}~~~ b_i = R\\
 -1~~~ \text{if}~~ b_i = L
 \end{cases} \hskip 2truecm
\end{equation}

\begin{equation} 
 n(i) := \begin{cases} \sum_{j=0}^{i-1} f(j)~~ \text{if $i$ is nonnegative}\\
  -\sum_{j=i}^{-1} f(j)~~ \text{if~~ $i$ is negative}
\end{cases}
\end{equation}

Note that $n(0)=0$.
\end{defn}

\begin{defn}\label{defn:coo}
 For each $i, X_{i}[1]$ and $X_{i}[2]$ are the first and second coordinates of $X_i$ resp.
\end{defn}
\begin{co}\label{co:rw}
 $X_{i}[1] = a_{n(i)}, X_{i}[2] = b_i.$
\end{co}
\begin{defn}\label{defn:ttp} In the \ttp , \proa is called the scenery at time 0 and \prob is called the path at time 0. $...A(-1),A(0),A(1),A(2),...$ is called the scenery at time $i$ if $A(k) = a_{n(i)+k}$ and the process $...B(-1),B(0),B(1),B(2),...$ is called the path at time $i$ if $B(k)=b_{i+k}$.\\ $...X_{-1}[1],X_{0}[1],X_{1}[1]...$ is called the scenery process which is distinct from the scenery at time $0$, but there is no distinction between the path process ($...X_{-1}[2],X_{0}[2],X_{1}[2]...$) and the path at time $0$. 
\end{defn}
\begin{co}\label{co:inttt}
INTUITIVE IDEA OF THE $TT^{-1}$ PROCESS: You should envisage the \ttp to be a random walk on a random scenery where $L$ means that you walk left, $R$ means you walk right and the heads and tails are the scenery.
\end{co}

\begin{defn} \label{defn:scesig} The \sal generated by just the knowledge of the $X_i[1]$ where $i$ runs over all the integers is called the scenery \sall .
\begin{co}\label{co:partitions}
By now the reader might be getting confused about the various partitions introduced here. Throughout the remainder of the text $P$ always means $P := ((h,L),(h,R),(t,L),(t,R)).~P$ is a refinement of two partitions. One is the path partition $(R,L)$ and the other is the scenery partition $(t,h)$. Whenever we want to refer to those two partitions we will explicitly refer to them as the $(R,L)$ partition and  the $(t,h)$ partition. If $T$ is the \tttl, the $T,P$ process is the \ttpl, the $T,(R,L)$ process is the path process and the $T,(t,h)$ process is the scenery process. The most confusing concept is the scenery at time 0 often referred to as simply the scenery. The trouble is that there is no partition $Q$ whatsoever such that the scenery at time 0 is the $T,Q$ process. You start with the scenery at time 0 to create an output of the \ttpl, but you can't get the scenery at time 0 as a standard factor of the \ttpl. \end{co}

\end{defn}

\begin{co}\label{co:omsce} If we just provide you with an \om it is clear what we mean by the scenery process for \oml . It is just the first coordinates of the $T,P$ name of \om where $T$ is the \tttl, i.e. it is the $T,(h,t)$ process. However, there is no apriori reason why you should know the scenery at time 0 for \om if all you know about \om is its $T,P$ name. But in fact, the $T,P$ name (or just the past $T,P$ name or just the future $T,P$ name) above does determine the scenery at time $0$ because random walk is recurrent so if you know the whole $T,P$ name then you can watch the random walk cover the entire scenery and thus you can observe it telling you exactly what that scenery is. Thus the following definition makes sense.
\end{co}
\begin{defn} \label{defn:omsce}The scenery at time $0$ for \om is the scenery at time $0$ determined by the $T,P$ name of \om where $T$ is the \ttt and $P = \{(h,L),(h,R),(t,L),(t,R)\}$. This is the distict from the scenery process at time $0$ which is just the first coordinates of that $T,P$ name, i.e. 

 $...X_{-2}(\omega)[1],X_{-1}(\omega)[1],X_{0}(\omega)[1],X_{1}(\omega)[1],X_{2}(\omega)[1],...$
\end{defn}
\begin{co}\label{co:nthcoor} While the scenery process and the scenery at time 0 are distinct, they both have the same $0^\text{th}$ coordinate. The $0^{\text{th}}$ coordinate for the output of the scenery process of \om at time 0 is $X_0[1]$ but here we used \om to define

 $. . .,X_{-2}, X_{-1}, X_0, X_1, X_2,. . .$ 

\noindent If we similarly took $T^{100}(\omega)$ and used that do define 

$. . .,X'_{-2}, X'_{-1}, X'_0, X'_1, X'_2,. . .$ 

\noindent  then $X'_0=X_{100}$  and hence $X'_0[1]=X_{100}[1]$. Thus the scenery process is the sequence of sequence of $0^\text{th}$ coordinates of the scenery processes of  

$...T^{-2}(\omega), T^{-1}(\omega), T^{0}(\omega), T^{1}(\omega),T^{2}(\omega),...$

\noindent and this implies (fixing \oml) that the scenery process can be thought of as the $0^\text{th}$ coordinates of the sceneries at times
... -2,-1,0,1,2... 
\end{co} 

\begin{co}\label{co:hoff} The purpose of \cite{hoff} by Hoffman was to establish that the scenery process is not loosely Bernoulli, a fact that is quite significant but irrelevant to this paper. However, in the process of proving that, Hoffman established the \sal C which we are about to describe. He starts by quoting a result of Matzinger \cite{mat} which states that if $T$ is the \tttl , there is a $Z$ of measure $0$ so that if $\omega_1 \notin Z$ and $\omega_2 \notin Z$, and \omo and \omt are in the same fiber of $A$ over the scenery sigma algebra (i.e. if $T^i(\omega_{1})$ and $T^i(\omega_{2})$ are in the same element of the two set partition $(h,t)$ for all $i \in \mathbb{Z}$) then either the scenery at time $0$ of \omo is an even translate of the scenery at time $0$ of \omt or the scenery at time $0$ of \omo is an even translate of the inverse of the scenery at time $0$ of the \omt  (i.e. if \proa is the scenery of \omo at time $0$ and $… a'_{-2}, a'_{-1}, a'_{0}, a'_{1}, a'_{2}…$ is the scenery of \omt at time $0$ then there exists an even integer $k$ such that either $a'_i = a_{i+k}$ for all $i\in \mathbb{Z}$ or $a'_i = a_{k-i}$ for all $i\in \mathbb{Z}$). Hoffman then got a two point extension of the scenery \sal such that, off a set of measure $0$, 

they are in the same fiber of $A$ over that two point extension

iff

\omo and \omt are in the same fiber of $A$ over the scenery \sal and there exists even integer $k$ such that $a'_i = a_{i+k}$ for all $i$.

He then went on to show that the two point extension he defined splits. 
\end{co} 
 
\begin{defn} \label{defn:C} $C$ is the two point extension of the scenery process that Hoffman defined above. In other words $C$ is a \sal with the following properties.\\
\\1)	There is a set $Z$ of measure 0 such that if $\omega_1 \notin Z$ and $\omega_2 \notin Z$, letting the scenery at time 0 for \omo be \proa and  the scenery at time $0$ for \omt be $… a'_{-2}, a'_{-1}, a'_{0}, a'_{1}, a'_{2}…$, \omo and \omt are in the same fiber of $A$ over $C$ iff they have the same name in the scenery process and there is an even $k$ such that $a'_i = a_{i+k}$ for all $i$.\\
\\2) $C$ splits.	\end{defn}

\begin{co}\label{co:Camb} Note that we are not giving an explicit definition of $C$ in this paper. That is because we don't need it. All we need to know about $C$ is (1) and (2) above.
\end{co}

\begin{defn}\label{defn:r1}
We put an equivalence relation, to be called equivalence relation $1$, on the set of all outputs of the \ttp by saying that \omo and \omt are equivalent1 if their outputs in the \ttp  have the same first coordinates, i.e. if they have the same name in the scenery process. 
\end{defn}
\begin{theorem}\label{thm:r1bn} Let $S$ be a set in $C$. Then there exists a set $Z$ of measure $0$ such that for any \omo and \omt if
\\1) neither \omo nor \omt are in $Z$.
\\2) \omo and \omt are equivalent1
\\3) the outputs of \omo and \omt in the \ttp only differ on finitely many coordinates.
\\
then either both \omo and \omt are in $S$ or neither of them is. 
\end{theorem}

\begin{proof}  
(3) implies that there is an $N$ such that the outputs of \omo and \omt are identical for all $i \leq N$ and that in turn implies that they have the same scenery at time $N~$(Generally one would want to think of $N$ as negative but it could be positive). Since the scenery at time $0$ is a translate of the scenery at time $N$, the sceneries at time 0 for the two of them are translates of each other. In fact they are even translates of each other because to get to the first to the scenery at 0 from the scenery process at time $N$ you have to take $|N|$ steps of size $1$ or $-1$ and to get to the second scenery at time 0 from the scenery at time $N$ you also have to take $|N|$ such steps. The result follows by (1) of \dref{C} and \tref{bn}
\end{proof}
\begin{defn}\label{defn:seen}Let $d<e$ be integers. Then $X_{d},X_{d+1}...X_{e}$ sees only a finite subset of the scenery \proa, i,e. there is a $j$ and a $k$ such that the walk from time $d$ to time $e$ only covers
$a_{j}, a_{j+1}, a_{j+2},... {a_k}.$~More precisely, j and k are the max and min of $\{n(i):d<=i<=e\}$ resp.,
where $n(i)$ is as in \dref{fN}.
\\$a_{j}, a_{j+1}, a_{j+2},... a_{k}$ is called the block of scenery seen from time $d$ to time $e$. 
\end{defn}
\begin{co}\label{co:block} "block of scenery seen from time $d$ to time $e$" is NOT the same as the scenery seen from time $d$ to time $e$ in the scenery process. The former is a concept of scenery in space and the latter is a concept of scenery in time.
\end{co}
\begin{defn}\label{defn:foba} Let $d<e$ be integers. Then $b_{d}, b_{d+1}, b_{d+2},... b_{e-1}$ is a portion of the path taking on values in $\{L,R\}$. Let $f$ be as in \dref{fN}. Then 
$fo := \max\limits_{-1\leq k\leq e-d-1}\sum\limits_{i=0}^{k} f(b_{d+i}) $ is called the forwards distance you walk from time $d$ to time $e$ and
$ba :=\min\limits_{-1\leq k \leq e-d-1}\sum\limits_{i=0}^{k} f(b_{d+i})$ is called the backwards distance you walk from time $d$ to time $e$. The net distance you walk from time $d$ to time $e$ is $\sum\limits_{i=0}^{e-d-1} f(b_{d+i})$. Here $\sum\limits_{i=0}^{-1}$(of anything) always takes on value 0.
\end{defn}

\begin{co}\label{co:fo-ba} Let $a_{j}, a_{j+1}, a_{j+2},… a_{k}$ be the block of scenery seen from time $d$ to time $e$.
Let $fo$ be the forwards distance you walk from time $d$ to time $e$.
Let $ba$ be the backwards distance you walk from time $d$ to time $e$.
Then $k-j= fo-ba+1.$
\end{co}

We now describe the \sal $B$ (a specific modification of Thouvenot's \sal).

\begin{defn}\label{defn:B}  Let S be the set of all \om whose first 11 outputs in the \ttp are 
~\\$(h,L),(h,L),(h,L),(h,R),(h,R)(h,L),(h,R),(h,L),(h,L),(h,L),(t,L)$.
~\\Let $P'$ be the partition of \Om which contains the complement of $S$ and the following partition of $S$. We let $\omega_{1}$ in $S$ and $\omega_2$ in $S$ be in the same piece of that partition iff, the following three conditions are satisfied:

\begin{align}
& \text{Let $m$, $n$ be the least positive integers such that}\label{eg:tho1} \\
& T^{-m}(\omega_1) \in S \text{~and~} T^{-n}(\omega_2) \in S. \text{~Then~} m=n.\notag \\
\notag \\
& \text{The net distance, forwards distance, and backwards distance from} \notag \\
& \text{time $-m$ until time 0 are the same for \omo and \omtl.}\label{eg:tho2} \\
\notag\\
& \text{The block of scenery seen from time $-m$ to time $0$ is the same} \notag\\
& \text{for \omo and \omt.}\label{eg:tho3}
\end{align} 
\end{defn}

\begin{defn}\label{defn:BB}
\noindent $BB$ is the literal \sal generated by $\{T^{i}P': i \in \mathbb{Z}\}$ and $B$ is the \sal generated by $\{T^{i}P': i \in \mathbb{Z}\}$, i.e. $S' \in B$ if it can be written as $SS' \Delta Z$ where $SS' \in BB$ and $Z$ has measure $0$.
\end{defn}

\begin{co}\label{co:intuitive}
INTUITIVE IDEA OF THE $B$: 
\\If you are in $S$, by \eref{tho1}, $P'$ tells you how long it has been since the previous time you were in $S$. Then by \eref{tho3}, $P'$ tells you the block of scenery you saw since the last time you  were in $S$ and by \eref{tho2}, $P'$ tells you where you were in that scenery the last time you were in $S$. Now look at the last time you were in $S$ and then $P'$ told you when you were in $S$ the previous time before that, what block of scenery you had seen between that time and the time before that and where you were in the block of scenery before that. Piecing this all together you can deduce from the entire $T,P'$ process (or just your past $T,P'$ process or just the future $T,P'$ process) your complete scenery because random walk is recurrent. If you are not in $S$ you cannot know exactly where you are in that scenery but since you did know the scenery the last time you were in $S$ and since the scenery at any time is just a translate of the scenery at any other time you know the scenery up to a translate. Furthermore you know how long it has been since the last time you were in $S$ which is an upper bound of how big that translate is. Since a fiber of the \sal generated by $T,P'$ tells you the future as well as the past you also know what the scenery will be the next time you enter $S$ and an upper bound for how much of a translate your scenery is from that scenery.
\end{co}
\begin{notation}\label{not:S} There is a certain ambiguity in our notation. Until now $S$ was a general $S$. Henceforth $S$ and $P'$ are as in \dref{B}.
\end{notation}
\begin{defn}\label{defn:r2} Let $\omega_{1}$ and $\omega_{2}$ be two elements of \Om. We say that $\omega_{1}$ and $\omega_{2}$ are equivalent2 if $T^{i}(\omega_{1})$ and $T^{i}(\omega_{2})$ are in the same atom of $P'$ for all integers $i$, i.e. if $\omega_{1}$ and $\omega_{2}$ are in the same atom of $T^{i}(P')$ for all integers $i$.
\end{defn}

Equivalent2 is equivalent to being in the same fiber of B (up to measure 0). The following is immediate from \tref{bn}

\begin{theorem}\label{thm:r2bn}Let $S_1 \in B$. Then there exists a $Z$ of measure $0$ such that if 
\\1) neither \omo nor \omt are in $Z$. 
\\2) \omo and \omt are equivalent2 . 
\\Then either both are in \omo and \omt are in $S_1$ or neither of them is.
\end{theorem}
 
\newpage
\section{Proof that $B$ and $C$ \\ have trivial intersection.}
 
\defn\label{defn:out} Henceforth $T$ is the $TT^{-1}$ transformation, 
~\\$P := \{(h,L),(h,R),(t,L),(t,R)\}$ and for \om in \Om, 
~\\
$...(a_{-2},b_{-2}),(a_{-1},b_{-1}),(a_{0},b_{0}),(a_{1},b_{1}),(a_{2},b_{2})...$, 
~\\
called the output of \om for $T$, is defined by $(a_{i},b_{i})$ is the atom of $P$ containing $T^i(\omega)$. 

\begin{co}\label{co:det}
Since $P$ is a generator of $T$, the output of \om completely determines \om
\end{co}
\begin{defn}\label{defn:ss} We suppose the existence of a set $SS$ such that  $SS \in B \cap C$.
\end{defn}
\begin{lemma}\label{lem:zz0}
 Let $Z$ be any set of measure $0$ and 
\\$ZZ = \{ \omega: \exists ~\omega_{1} \in Z$ whose output (in the \ttpl) differs from that of \om in only finitely many coordinates.\}
\\Then $ZZ$ has measure $0$.\end{lemma}

\begin{proof} Let \omb be the set of paths of the \ttp with the appropriate measure on it. Fix a finite set of integers $F$. We now define a measure preserving isomorphism $S_F$ from \omb to itself. Let $Le$ be the least element of $F$. An output of the \ttp is the scenery process and the path process but in fact if you just know the output of the \ttp before a given time and just the path process from that time onward then you know the entire \ttp output because the output before a given time determines the scenery at that time and therefore the path from then onward will determine scenery process from then onward. We define a transformation $S_F$  from the space of \ttp outputs to itself by letting $S_F(\omega)$ be the unique output of the \ttp which is identical to that of \om before $Le$ and whose path is identical to that of \om from $Le$ onward except the opposite of the path of \om (i.e. switch $L$ to $R$ or $R$ to $L$.) on those times in $F$. $S_F$ can easily be seen to be measure preserving so $S_F(Z)$ has measure $0$. $ZZ \subset \bigcup\limits_{\text{all finite sets}~ F} (S_F(Z))$.\end{proof}
\begin{lemma}\label{lem:iff}There exists $ZZ$ of measure $0$ such that if 
the outputs of~\omo and \omt in the \ttp differ on only finitely many coordinates and $\omega_1 \notin ZZ$  then $\omega_1 \in SS$ iff $\omega_2 \in SS$.
\end{lemma}
\begin{proof}
By \tref{r1bn} and \tref{r2bn} there exists a set $Z$ of measure $0$ such that for any two points which are equivalent1 or equivalent2, whose output in the \ttp differs on only finitely many terms, neither of which are in $Z$, either both are in $SS$ or neither of them is. Let

\vspace{.5cm}
\noindent $ZZ=\{\omega: \exists \omega_{11} \in Z$ whose output differs from \om in only finitely many coordinates$\}$ 

\vspace{.5cm}
\noindent so that by \lref{zz0} ~$ZZ$ has measure $0$. Note that for any $\omega_{11}, \omega_{12}$, if  \om  $\notin ZZ$ and $\omega_{11}$ and $\omega_{12}$ differ from \om in only finitely many coordinates, then $\omega_{11} \notin Z$ and $\omega_{12} \notin Z$. Hence the proof will be complete when we can show that for any outputs \omo and \omt of the \ttp which differ on only finitely many coordinates, there exists $\omega_3$ and $\omega_4$ such that 
~\\1) any two elements of  $\{\omega_1,\omega_{2},\omega_{3},\omega_{4}\}$ differ on only finitely many coordinates.
~\\2) $\omega_{1}$ and $\omega_{3}$ are equivalent1. 
~\\3) $\omega_{2}$ and $\omega_{4}$ are equivalent1.
~\\4) $\omega_{3}$ and $\omega_{4}$ are equivalent2.
~\\
~\\

Select $M$ such that \omo is the same as \omt outside of $[-M,M]$. Note that for any $a<-M$ and $b>M$, \omo and \omt see identically the same scenery at both $a$ and $b$ because for $a$ they see exactly the same present and past and the present and past of the \ttp determines the scenery; for $b$ they see exactly the same present and future which also determines the scenery.  Using the recurrence of random walk and the fact that the \omo walk is identical to the \omt walk before $-M$ select $N_1<-M$ so that both the \omo walk and $\omega_{2}$ walk from $N_1$ to $M$ reaches its lowest point at time $N_1$. Select $N_2>M$ so that both the \omo walk and $\omega_{2}$ walk from $N_1$ to $N_2$ reaches its highest point at time $N_2$. We also want to make sure that 
~\\
~\\5) at neither times $N_1$ nor $N_2$ we are in the middle of a stretch of the form
~\\$(h,L),(h,L),(h,L),(h,R),(h,R)(h,L),(h,R),(h,L),(h,L),(h,L),(t,L)$

\vspace{.5cm}

It is intuitively obvious that we should be able to do insist on 5) but one explicit way to do it is to realize that if we drop condition 5) there are infinitely many values $n$ we could have used for $N_1$ so with probability $1$ one such $n$ has $24$ successive $R$s in a row preceding it and then just let $N_1=n-12$. Argue symmetrically to choose $N_2$.

It follows that for any $a \leq N_1$ and any $b \geq N_2$ the backwards, forwards and net distances from $a$ to $b$ is the same for both \omo and \omtl . Alter \omo to get $\omega_{3}$ by changing each

$(h,L),(h,L),(h,L),(h,R),(h,R)(h,L),(h,R),(h,L),(h,L),(h,L),(t,L)$
~\\between $N_1$ and $N_2$ to 

$(h,L),(h,L),(h,L),(h,R),(h,L)(h,R),(h,R),(h,L),(h,L),(h,L),(t,L)$
~\\and $\omega_3$ is a legitimate output of the \ttpl. The reason we added $(t,L)$ at the end of 

$(h,L),(h,L),(h,L),(h,R),(h,R)(h,L),(h,R),(h,L),(h,L),(h,L),(t,L)$
~\\when we defined $S$ was to guarantee that no such change would create a new instance of $S$. $\omega_1$ and $\omega_{3}$ are equivalent1. Do the same to $\omega_{2}$ to get a $\omega_{4}$ such that $\omega_{2}$ is equivalent1 to $\omega_{4}$. If $a \leq N_1<N_2 \leq b$ none of these changes alter the forwards, backwards, or net distances from $a$ to $b$. $\omega_{3}$ and $\omega_{4}$ have no instance of $S$ between $N_1$ and $N_2$ so they are equivalent2 because they see $S$ at the same times, the forwards, backwards, and net distances between the previous times less than $N_1$ they saw $S$ and the next time greater than $N_2$ that they see $S$   and the sceneries they see at both those times are the same.
\end{proof} 
\begin{co}\label{co:done}In principle we are already done. \lref{iff}  establishes that $SS$ is in the double tailfield of the \ttp which is trivial so $SS$ is trivial. However we prefer to give a more self contained proof that does not refer to the double taifield of the \ttpl.
\end{co}
\begin{lemma}\label{lem:couple}Select an $N>0$ and let 
~\\$G = (a_{-N},b_{-N}),...(a_{-1},b_{-1}), (a_0,b_0), (a_1,b_1),… (a_N,b_N)$
and 
~\\$H= (a'_{-N},b'_{-N}),...(a'_{-1},b'_{-1}), (a'_0,b'_0), (a'_1,b'_1),… (a'_N,b'_N)$
~\\be two words in the \ttp from $-N$ to $N$. Let $\mu$ and $\nu$ be the conditioned measures of the \ttp conditional on $G$ and $H$ resp.  Then there is a coupling of  $\mu$ and $\nu$ (Coupling $\mu$ and $\nu$ means a measure on the product space where $\mu$ and $\nu$ are the marginals. For example, a joining is a special case of a coupling) such that with probability $1$, a pair of paths in the coupled process differs only on a finite set.
\end{lemma}
\begin{proof}Pick a doubly infinite output, $p =...(a_{-1},b_{-1}), (a_0,b_0), (a_1,b_1),...$ in accordance with measure $\mu$. Consider the scenery $sce$ at time $0$ determined by $p$. In that scenery look at the $w$ word from $-N$ to $N$ ($w$ is the actual scenery from $-N$~ to ~$N$ of size $2N+1$, not just the scenery you get to see in $G$ which is usually on the order about $\sqrt{N}$ in size). Let $w'$ be the part of the scenery that $H$ tells you, which is typically on the order of $\sqrt{N}$ in size, part of the past scenery and part of the future scenery. Extend $w'$ randomly to a arbitrary scenery word $w_1$ from $-N$ to $N$ where we let every possible extension of $w'$ have equal probability of being the $w_1$ that we choose. Now in $sce$, let $M$ be the least integer such that $M>N, M+N$ is an even integer, and $w_1$ is the word from $M$ to $M+2N$. We now define a doubly infinite output $q$ in the \ttpl. One way to pick an output is to tell you the scenery at time $0$ and the path at time $0$ separately and then let them generate $q$. We will let the scenery of $q$ at time 0 be the scenery of $p$ at time 0 translated to the right by  $M+N$ so that the scenery at time $0$ for $q$ from $-N$ to $N$ is precisely $w_1$. We let the path from time $-N$ to time $N$ be $b'_{-N}, ... b'_{-2}, b'_{-1}, b'_0, b'_1, b'_2,... b'_N$. Since $w_1$ is an extension of $w'$ this forces the output in the \ttp for $q$ from time $-N$ to time $N$ to be precisely $H$. To get the path from time $N+1$ onward, just run it independently of the path of $p$ until 
\\ the amount of $R$s - the amount of $L$s in $p$ 
\\ exceeds 
\\the amount of $R$s - the amount of $L$s in $q$ 
\\by $M+N$. This will eventually happen with probablilty $1$ because the difference between two random walks is a recurrent random walk on the even integers. When that happens $p$ and $q$ will see the exact same scenery so from then on couple them to be exactly the same.
Using a symmetric argument run the path of $q$ at times $-(N+1),-(N+2)...$ to be independent of the path of $p$ until they see the same scenery and then let them be identical from then on.
We claim that our random selection of $(p,q)$ is the desired coupling. It is a measure on the product space of the \ttp with itself where with probability $1$, $p$ and $q$ differ on only finitely many  terms and $p$ has been chosen in accordance with $\mu$. We need only show that $q$ ends up being chosen in accordance with measure $\nu$. The proof will be complete if we can choose another coupling of $p$ and $q$ in which we first select $q$ in accordance with measure $\nu$ and then choose $p$ from $q$ in such a way that in the end the two couplings turn out to be identical. Here we get $p$ from $q$ in exactly the same way that we got $q$ from $p$ except instead of picking $M$ to be the least time bigger than $N$ with the appropriate property we let $-M$ be the largest time below $-N$ with the appropriate property and we produce the coupling in the same way(except translate forward by $M+N$ instead of backward). We think it is clear that we can get the exact same coupling that way, i.e that the measures which chose $p,q$, are identical in these two couplings.
~\\ \indent It should be intuitively obvious that these two couplings are identical but to see it more rigorusly, first note that in both cases the measure on the pair \oml,\omo are the same where \om is the $\mu$ word of length $2N+1$ and \omo is the $\nu$ word of length $2N+1$. The probability law on $M+N$ in the first coupling is identical to that of the $-(M+N)$ in the second (and furthermore is independent of the pair \oml,\omol) and for any k, the distribution of everything else(sceneries and paths) are identical given \oml, \omol, $M+N=k$ for the first coupling and $-(M+N)=k$ for the second coupling.
\end{proof}
\begin{theorem}\label{thm:BC}
$B$ and $C$ have trivial intersection.
\end{theorem}
\begin{proof}
~\\ \lref{iff} and \lref{couple} establish that there is a $Z$ of measure $0$ such that for $\mu$ and $\nu$ as in \lref{couple} when $p,q$ is picked as in the coupling of $\mu$ and $\nu$ given by that lemma and $p\notin Z, p \in SS$ iff $q \in SS$. It follows that the $\mu$ measure of $SS$ is the same as the $\nu$ measure of $SS$ and hence that (Regarding $G$ and $H$ of \lref{couple}) the measure of $SS$ given $G$ is the same as the conditional measure of $SS$ given $H$. Since those are any two words from $-N$ to $N$ for any $N$ it follows that $SS$ is independent of any cylender set. Since $SS$ can be arbitrarily well approximated by a cylender set it follows that the measure of $SS$ is either $0$ or $1$.
\end{proof}
\newpage
\section{Proof that $B$ splits}
\begin{theorem}\label{thm:split} $B$ splits\end{theorem}
\begin{proof}In \cite{vwb}, Thouvenot established a condition for determining whether or not a factor splits. The Thouvenot condition (slightly rephrased) is the following generalization of the very weak Bernoulli condition (actually it is a slightly rephrased form of the relative very weak Bernoulli condition which in \cite{vwb} was shown to imply the relatively finitely determined condition which in \cite{A} was shown to imply splitting). 

\vspace{.5cm}

$<$Thouvenot condition: Let $T$ be a transformation, let $Q_1$ be a partition which generates $T$, and let $Q_2$ be another partition. If for every $\epsilon>0$, for all sufficiently large $n$, there is a coupling of  the $((Q_1 \bigvee Q_2),T)$ process with itself so that for (\omol,\omtl) in the coupled process. 

\vspace{.5cm}

\noindent 1)	The $Q_2$ name of \omo is the same as the $Q_2$ name of \omt.
\\2)	The past $Q_1$ name of \omo conditioned on the complete $Q_2$ name is independent of the past $Q_1$ name of \omt conditioned on the complete $Q_2$ name.

\noindent 3) For $\epsilon>0$ and all sufficiently large $n$, the expected mean hamming distance between
 ~\\
 ~\\ the $Q_1$ name from time $0$ to time $n$ of \omo  
~\\
~\\and 
~\\
   ~\\ the $Q_1$ name from time $0$ to time $n$ of \omt
    ~\\
    ~\\is less than $\epsilon$.

\vspace{.5cm}
  
\noindent Then it follows that the $Q_2,T$ process splits, which means that there is another partition $B$ such that the $B,T$ process is independent of the $Q_2,T$ process, the $B,T$ process is Bernoulli, and $Q_2 \bigvee B$ generates $T>$

\vspace{.5cm}
 
Hence we only have to establish $1,2$ and $3$, when $Q_2$ is the $P'$ of  \dref{B} and $Q_1$ is the standard generator $P$ of the \ttpl, namely $P:=\{(h,L),(h,R),(t,L),(t,R)\}$.
We now start to produce the required coupling. We can just assume (1) and (2) by simply picking a $P'$ name in accordance with the measure on $P'$ names and then couple the past $P$ names independently given that $P'$ name. 

Now we have started the coupling. We have coupled the complete $P$ and $P'$ pasts and the $P'$ name of the future. What is left is to couple the complete future $P$ names (given the pasts and the $P'$ futures) so that they are close in mean hamming distance from time 0 to time $n$ for sufficiently large $n$. In fact we will do something stronger, we will get the future $P$ names to completely agree eventually. Before continuing the coupling, let us see what we already know.
\begin{align}
&\text{a)	We know the $P'$ names of the past.}\notag\\
&\text{b)	We know that the $P'$ names of the past are the same.}\notag\\
&\text{c)	We know the future $P'$ names.}\notag\\
&\text{d)	We know that the future $P'$ names are the same.}\notag\\
&\text{e)	We know the past $P$ names.}\notag
\end{align}

By a) and \coref{intuitive}, the first positive time either process lands in $S$ they will know their entire scenery and by b), c) and d) the first time the processes land in $S$ will be the same and the scenery they will both know at that time will be the same. Suppose that the first positive time they land in $S$ will be time 30. The \ttt  has been defined by starting with the path and the scenery at time 0 and use them to generate an element of the \ttpl. Hence from time 30 onward, since you know the scenery at time 30,  the future (after time 30) path will determine the $P$ and $P'$ name but given the scenery, the future path will determine the future $P'$ name INDEPENDENT OF THE PAST $P'$ NAME AND PAST $P$ NAME. Keep in mind that when one parameter determines another then knowledge about the latter parameter tells you information about the former parameter and nothing else, e.g. if your name determines whether or not you will be wealthy, then knowing that you will be wealthy determines information about your name and it does not determine anything else. So given the scenery, the future $P'$ name gives information about the future path after time 30 and about nothing else and the information it gives will be the same for both \omo and \omt because they see the same scenery. Hence we can continue the coupling by coupling independently until the first time the two processes reach $S$ (in our example until time 30) and then couple the two paths to be the same after that (which causes the $P$ names to be the same after that)
\end{proof}


\newpage

\begin{thebibliography}{9}
\bibitem{OS}Ornstein, Donald S.; Shields, Paul C. An uncountable family of $K$-automorphisms. {\em Advances in Math.} (10), 63--88. (1973). MR0382598 {\bf52:}3480
28A65


\bibitem{ttt}S. Kalikow. \ttt is not loosely Bernoulli, {\em Annals of Math.} (2) {\bf115} (1982), 154--160. MR {\bf84j:}28023

\bibitem{A}J Thouvenot. Quelques proprietes des systemes dynamiques qui se decomposent en 
un produit de deux systemes dont l'un est un schema de Bernoulli.
Israel Journal of Mathematics vol 21, 2-3 (1975)

\bibitem{again} J Thouvenot. Two facts concerning the transformations which satisfy the weak 
Pinsker property.
Ergodic Theory and Dynamical Systems 28 (2008)

\bibitem{hoff}Hoffman, Christopher.; The scenery factor of the $[T,T\sp {-1}]$ transformation is not loosely Bernoulli.{\em Proc. Amer. Math. Soc}. {\bf131} (2003),  no. 12, 3731--3735 (electronic). MR1998180 (2004j:60203)  60K37 (28D05 37A35)

\bibitem{mat}H.Matzinger. Reconstructing a 2 color scenery by observing it along a simple random walk path {\em ANN.Appl. Probab.} {\bf15}(2005) no. 1B, 778--815 (Reviewer: Fabio P. Machado) \underline{60K37(60G50)}

\bibitem{vwb} J Thouvenot. Remarque sur les systemes dynamiques donnes avec plusieurs facteurs
Israel Journal of Mathematics vol 21, 2-3 (1975)

\end{thebibliography}
\end{document}